\documentclass[journal]{IEEEtran}
\usepackage{amssymb,amsmath,amsthm,amsfonts}
\usepackage{color}
\usepackage{epstopdf}
\usepackage{graphicx}
\usepackage{subfigure}

\def\dref#1{(\ref{#1})}
\def\disp{\displaystyle}
\newtheorem{remark}{Remark}[section]
\newtheorem{theorem}{Theorem}[section]                   
\newtheorem{definition}{Definition}[section]                   

\newtheorem{lemma}{Lemma}[section]

\ifCLASSINFOpdf
\else
\fi
\begin{document}

\title{Linear and Nonlinear Event-Triggered Extended State Observers for Uncertain Stochastic Systems}

\author{Ze-Hao Wu, Feiqi Deng, Hua-Cheng Zhou, and  Zhi-Liang Zhao
\thanks{Ze-Hao Wu is with School of Mathematics and Big Data, Foshan University, Foshan 528000, China.
        {\tt\small  Email: zehaowu@amss.ac.cn} }
         \thanks{Feiqi Deng  is with Systems Engineering Institute, South China University of Technology, Guangzhou 510640, China.
 {\tt\small  Email: aufqdeng@scut.edu.cn} }
 \thanks{Hua-Cheng Zhou  is with School of Mathematics and Statistics, Central South University, Changsha 410075, China.
 {\tt\small  Email: hczhou@amss.ac.cn} }
 \thanks{Zhi-Liang Zhao  is with School of Mathematics and Information Science, Shaanxi Normal University, Xi'an 710119, China.
 {\tt\small  Email: gsdxzzl@mail.ustc.edu.cn} }
}%

%
%

\markboth{}%
{Shell \MakeLowercase{\textit{et al.}}: Bare Demo of IEEEtran.cls
for Journals}
%



\maketitle

\begin{abstract}
In this paper, linear and nonlinear event-triggered extended state observers are designed for
a class of uncertain stochastic systems driven by bounded and colored noises. Two
 event-generators with an ensured positive minimum inter-event time for
every sample path solution of the stochastic systems, are proposed for the designs of
linear and nonlinear  event-triggered extended state observers, respectively.
 The mean square and almost sure convergence of the estimation errors of
unmeasured state and stochastic total disturbance including internal uncertainty and external stochastic noises is
presented with rigorous theoretical proofs. Compared with the linear event-triggered extended state observer, the theoretical results show that the nonlinear one
via homogeneity possesses higher estimation accuracy but is at the price of higher triggering frequency. Some numerical simulations are
performed to authenticate the theoretical results.
\end{abstract}

\begin{IEEEkeywords}
Stochastic systems, extended state observer, event-triggering mechanism, estimation, colored noise.
\end{IEEEkeywords}

%
\IEEEpeerreviewmaketitle

\section{Introduction}
%
%
%
%

\IEEEPARstart{O}wing to the ubiquity of uncertainties and disturbances in practical engineering systems,
disturbance rejection for uncertain systems has been one of mainstream topics in recent decades
 in the control community.
Active disturbance rejection control (ADRC) \cite{Han}, a highly novel  control technology
based on estimation/compensation strategy, can cope with  uncertainties and disturbances in large scale. The
significance of ADRC has been validated by considerable number of engineering applications.
The key component part of ADRC is the extended state observer (ESO) designed to be linear or nonlinear, aiming at
real-time estimation of not only unmeasured state but also total disturbance (extended state) affecting system performance;
Based on the estimates obtained via ESO, an active anti-disturbance control constructed by a feedback control and a compensator can be
designed for disturbance rejection and prospective control objective. Therefore, the theoretical foundation for the convergence of ESO is extremely important
 in ADRC's theory and applications.

The convergence of linear and common nonlinear ESOs for uncertain systems has been investigated in the past two decades,
see \cite{zhaoESO2011,zhaofrac}
  and the references therein. The convergence of linear ESO and
nonlinear ESOs based on an exponentially stable system and finite-time stable one
for uncertain stochastic systems have also been researched in \cite{wu1,wuMIMOESO}.
It should be noted that most of these works concerning the design and theoretical analysis of ESO
are based on continuous-time  output  measurement.
However, in prevailing networked control systems or digital control systems, one of the crucial  problems
 is to reduce the utilization of communication/computation  resources.  This greatly spurs the current research of
event-triggering mechanism (ETM) in the estimation and control of systems, primarily on account of
its advantage in saving communication/computation resources
(see, e.g., \cite{eventESO,liuy2019,deng2019,liuy2020}
 and references therein).

The convergence of a nonlinear event-triggered ESO for a class of uncertain systems has been developed in \cite{eventESO},
where the nonlinear gain functions are constructed by an exponentially stable system and cover the linear ones as a special case.
Nevertheless, it is admittedly more realistic that uncertainties and disturbances are often stochastic in practice.
One of the momentous problems for the ETM to be practical and feasible is to avoid the Zeno phenomenon (i.e.,
the triggering conditions are satisfied infinite times in finite time), which would bring challenging obstacle in developing the
 event-triggered ESO for uncertain stochastic systems.
This is because when there exist stochastic noises, each sample path of system state
may change variously even for the deterministic initial condition;
Specifically, the execution/sampling times and
the inter-execution times depending on sample path, are stochastic but not deterministic, so that
obtaining a positive lower bound for the stochastic inter-execution times (i.e., excluding Zeno phenomenon) is very sophisticated or even impossible.
Recently, novel ETM  with dwell time (time-regularization) and periodic ETM have been proposed for stochastic systems,
which force the ETM to allow a positive minimum inter-event time, see, e.g., \cite{liuy2019,deng2019,liuy2020}. The Zeno phenomenon can
be directly excluded in these design frameworks, while the theoretical analysis becomes much more difficult.

On the other hand, compared with the nonlinear ESO in \cite{eventESO} that is an almost linear one,
a more widely used nonlinear ESO based on a finite-time stable system via homogeneity has been shown to be with higher estimation accuracy
and better noise-tolerant performance \cite{zhaofrac,wuMIMOESO,zhaonlinlieaTD}; The convergence of the nonlinear event-triggered ESO via homogeneity for
uncertain systems is still unsolved, where the stochastic counterpart is more general and complex.

Motivated by aforementioned research status, in this paper we develop both the linear and  nonlinear event-triggered ESOs
for a class of uncertain stochastic systems driven by bounded and colored noises. The main
contributions and novelties can be summed up as follows: a) The uncertain stochastic systems
are subject to large-scale stochastic total disturbance which is the
total nonlinear coupling effects of unmodeled dynamics, external deterministic disturbance, bounded noise and colored noise;
b) Novel linear/nonlinear event-triggered ESOs are designed for the uncertain stochastic systems, and the mean square and almost sure convergence in the transient  process
is presented with rigorous theoretical proofs; c) The theoretical results reveal that the nonlinear event-triggered ESO via homogeneity
is with higher estimation accuracy but higher triggering frequency compared with the linear one.

This paper will be proceed as below. The problem formulation and some preliminaries are presented in Section
II. The designs of linear and nonlinear event-triggered ESOs and convergence results are given
in Section III. Some numerical
simulations are performed to validate the rationality of the
theoretical results in Section IV, with the
concluding remarks be followed up in Section V. For the
sake of readability, theoretical proofs are arranged in
Appendix A and Appendix B.

\section{Problem formulation and preliminaries}\label{Se2}

The following notations are used throughout the paper.
$\mathbb{E}$  denotes the mathematical expectation;
$|Z|$ represents the absolute value of a scalar $Z$, and
$\|Z\|$ represents the Euclidean norm of a vector $Z$;
$I_{n}$
denotes the $n$-dimensional identity matrix;
$\lambda_{\min}(Z)$  and  $\lambda_{\max}(Z)$ represent
the minimum eigenvalue and maximum  eigenvalue of a positive definite matrix $Z$;
$a\wedge b:=\min\{a,b\}$;
$\langle\vartheta\rangle^{s}:={\rm sign}(\vartheta)|\vartheta|^{s}$
for  all $\vartheta\in\mathbb{R}$ and any $s>0$; $\mathbb{I}_{\Omega_{0}}$ denotes an indicator function
with the function value being $1$ in the domain $\Omega_{0}$ and being $0$ otherwise;
$0_{s}$ represent the $1\times s$ row vector with all components to be zero.

To deal with the convergence of the nonlinear ESO constructed via homogeneity,
the definitions of  homogeneity
and some preliminary lemmas are introduced as follows.

 \begin{definition}(\cite{Perruq})\label{def1}
 The zero equilibrium of the following system
\begin{equation}\label{system1}
\dot{\vartheta}(t)=\psi(\vartheta(t)),\ \vartheta(0)=\vartheta_0\in \mathbb{R}^s
\end{equation}
with $\psi\in C(\mathbb{R}^{s};\mathbb{R}^{s})$ and $\psi(0_{s})=0_{s}$,
is said to be globally  finite-time stable, if the zero equilibrium of
system \dref{system1} is Lyapunov stable and for any
 $\vartheta_0\in\mathbb{R}^s$, there exists $T(\vartheta_0)>0$ such that the solution of
\dref{system1} satisfies that $\lim_{t\rightarrow T(\vartheta_0)}\vartheta(t)=0_{s}$  and
$\vartheta(t)=0_{s}$ for all $ t\in[T(\vartheta_0),\infty)$.
\end{definition}

\begin{definition}(\cite{Rosier})
 A function $W:\mathbb{R}^s\to\mathbb{R}$ is said to be homogeneous of degree $\theta$ with respect to weights
 $\{w_l>0\}_{l=1}^s$, if $W(\lambda^{w_1}\vartheta_1,\lambda^{w_2}\vartheta_2,\cdots,\lambda^{w_s}\vartheta_s)=\lambda^\theta W(\vartheta_1,\vartheta_2,\cdots,\vartheta_s)$
 for all $\lambda>0$ and all $(\vartheta_1,\cdots,\vartheta_s)\in \mathbb{R}^{s}$.
A vector field $W:\mathbb{R}^s\to\mathbb{R}^s$ is said to be homogeneous of degree $\theta$ with respect to weights
 $\{w_l>0\}_{l=1}^s$,   if for all $l=1,\cdots,s$, the $l$-th
 component  $W_l$ is a homogeneous function of degree $\theta+w_{l}$, that is,
$W_l(\lambda^{w_1}\vartheta_1,\lambda^{w_2}\vartheta_2,\cdots,\lambda^{w_s}\vartheta_s)=\lambda^{\theta+w_l}W_l(\vartheta_1,\vartheta_2,\cdots,\vartheta_s)$
 for all $\lambda>0$ and all  $(\vartheta_1,\cdots,\vartheta_s)\in \mathbb{R}^s$. The system \dref{system1} is homogeneous of degree $\theta$ if the vector
field $\psi$ is homogeneous of degree $\theta$.
  \end{definition}

 \begin{lemma}(\cite[Theorem 2]{Rosier} or \cite[Theorem 6.2]{Bhat2}) \label{fdf}
If system \dref{system1} is homogeneous of degree $\theta$ with
weights $\{w_{l}\}_{l=1}^{s}$,
and its zero equilibrium is globally asymptotically
stable, then for any $\mu>\max_{1\leq l\leq s}\{-\theta,w_l\}$,
 there exists a positive definite,
radially unbounded function $W\in C^{1}(\mathbb{R}^{s};\mathbb{R})\cap C^{\infty}(\mathbb{R}^{s}\setminus\{ 0_{s}\};\mathbb{R})\ $
 such that $W$ is homogeneous of degree $\mu$ with respect to weights $\{w_{l}\}_{l=1}^{s}$,
and the Lie derivative of $W(\vartheta)$ along the vector field $\psi$: $L_{\psi}W(\vartheta) :=
\nabla W(\vartheta)\psi(\vartheta)$ is negative definite.
\end{lemma}

\begin{lemma}\cite[Lemma 4.2]{Bhat2} \label{lemma2.1}
 Let $W_1,W_2:\mathbb{R}^s\to\mathbb{R}$ be continuous functions, homogeneous of degree $\theta_1>0,\theta_2>0$ with respect to the same weights
respectively, and $W_1$ is positive definite. Then for each $\vartheta\in\mathbb{R}^s$, it holds that
$(\min_{ \{z\in\mathbb{R}^{s} :\;W_1(z)=1\}}W_2(z))(W_1(\vartheta))^{\frac{\theta_2}{\theta_1}}\le W_2(\vartheta)
\leq  (\max_{\{z\in \mathbb{R}^{s} :\;W_1(z)=1\}}W_2(z))(W_1(\vartheta))^{\frac{\theta_2}{\theta_1}}$.
\end{lemma}

\begin{lemma}\cite[Lemma 8]{Perruq}\label{lemma2.2}
If $\nu \in(1-\frac{1}{n},1)$, then the vector field
\begin{eqnarray}\label{functionchi}
\hspace{-1.2cm}&&\Phi(\vartheta):=
(\Phi_{1}(\vartheta),\cdots,\Phi_{n}(\vartheta),\Phi_{n+1}(\vartheta))
\cr \hspace{-1.2cm}&&:=(\vartheta_2-a_{1}\left\langle\vartheta_1\right\rangle^{\nu},\cdots,\vartheta_{n+1}-a_{n}\left\langle\vartheta_1\right\rangle^{n\nu-(n-1)},
\cr\hspace{-1.2cm}  &&-a_{n+1}\left\langle\vartheta_1\right\rangle^{(n+1)\nu-n}), \forall \vartheta=(\vartheta_1,\cdots,\vartheta_{n+1})\in\mathbb{R}^{n+1}
\end{eqnarray}
is homogeneous of degree $\nu-1$ with respect to weights
$\{(l-1)\nu-(l-2)\}_{l=1}^{n+1}$.
\end{lemma}
Define
\begin{equation}\label{matricf2}\disp
G=
\begin{pmatrix}
-a_{1}    & 1      &0      &  \cdots   &  0             \cr -a_{2}
& 0 & 1   &  \cdots   &  0             \cr \cdots &\cdots & \cdots &
\cdots& \cdots  \cr -a_{n}     & 0     & 0     &  \ddots   &  1 \cr
-a_{n+1} & 0 &0 & \cdots & 0
\end{pmatrix}_{(n+1)\times (n+1)}.
\end{equation}
\begin{lemma}\cite[Theorem 10]{Perruq}\label{lemma2.3}
For $\nu\in (1-\frac{1}{n},1)$, if the  matrix $G$ in
 (\ref{matricf2}) is Hurwitz, then the system $\dot{\vartheta}(t)=\Phi(\vartheta(t))$ is
globally finite-time stable.
\end{lemma}

\begin{lemma}(The multi-dimensional It\^{o}'s formula) \cite[p.36, Theorem 6.4]{mao}
Let $x(t)$ be a $n$-dimensional It\^{o} process on $t\geq 0$ with the stochastic differential
$dx(t)=f(t)dt+g(t)dB(t)$,
where $f\in L^{1}([0,\infty);\mathbb{R}^{n})$, $g\in L^{2}([0,\infty);\mathbb{R}^{n\times m})$ and
$B(t)$ is a $m$-dimensional standard Brownian motion. Let
$V\in C^{2}(\mathbb{R}^{n};\mathbb{R})$. Then $V(x(t))$ is again an It\^{o} process  with the stochastic
differential given by
\begin{eqnarray}\label{LVdefin}
&&\hspace{-0.8cm}dV(x(t))=[\frac{\partial V(x(t))}{\partial x}f(t)+\frac{1}{2}\mbox{Tr}\{g^{\top}(t)\frac{\partial^{2} V(x(t))}{\partial x^{2}}g(t)\}]dt
\cr&&\hspace{1cm} +\frac{\partial V(x(t))}{\partial x}g(t)dB(t)\;
\cr&&\hspace{0.7cm} =:\mathcal{L}V(x(t))dt+\frac{\partial V(x(t))}{\partial x}g(t)dB(t)\;\;\; a.s.
\end{eqnarray}
\end{lemma}
\begin{remark}
The definition of $\mathcal{L}V(x(t))$ in \dref{LVdefin} will be used throughout the paper. It should be pointed out that we could not define the infinitesimal generator $\mathcal{L}V$
itself because  some random coefficients brought by bounded and colored noises such as aforementioned $f(t)$ and $g(t)$  occurring in the stochastic differential.
This would not cause any obstacle in the following theoretical analysis because we only use  $\mathbb{E}\mathcal{L}V(x(t))=\frac{d\mathbb{E}V(x(t))}{dt}$
when $\int^{t}_{0}\frac{\partial V(x(s))}{\partial x}g(s)dB(s)$ is a martingale for each $t\geq 0$.
\end{remark}

Let $(\Omega,\mathcal{F},\mathbb{F}, P)$ be
a complete filtered probability space with a filtration
$\mathbb{F}=\{\mathcal{F}_{t}\}_{t\geq 0}$  on which
two  mutually independent one-dimensional standard Brownian motions $B_{i}(t)\;(i=1,2)$ are defined.
In this paper, we consider a class of uncertain stochastic systems driven by bounded and colored noises as follows:
\begin{equation}\label{system1.2}
\left\{\begin{array}{l} \dot{x}_{1}(t)=x_{2}(t), \cr
\dot{x}_{2}(t)=x_{3}(t), \cr \hspace{1.1cm} \vdots \cr
\dot{x}_{n}(t)=f(t,x(t),v_{1}(t),v_{2}(t))+u(t),\cr
y(t)=x_{1}(t),
\end{array}\right.
\end{equation}
where $x(t)=(x_{1}(t),\cdots,x_{n}(t))\in \mathbb{R}^{n}$,
$u(t)\in \mathbb{R}$ and $y(t)\in \mathbb{R}$ are the state,
control input and output measurement, respectively; $f:[0,\infty)\times \mathbb{R}^{n+2}\rightarrow \mathbb{R}$
is an unknown system function satisfying Assumption (A1); $v_{1}(t):=\sigma(t,B_{1}(t))$ defined by
an unknown bounded function $\sigma:[0,\infty)\times \mathbb{R}\rightarrow \mathbb{R}$ satisfying Assumption (A2) is the bounded noise,
and $v_{2}(t)$ is the colored noise that is the solution to an It\^{o}-type stochastic differential equation
(see, e.g., \cite[p.426]{colorednoise}, \cite[p.101]{mao}):
\begin{equation}\label{21equ}
dv_{2}(t)=-\alpha_{1} v_{2}(t)dt+ \alpha_{1} \sqrt{2\alpha_{2}}dB_{2}(t),
\end{equation}
where $\alpha_{1}>0$ and $\alpha_{2}>0$ are constants representing
the correlation time and the noise intensity, respectively.
It is noteworthy that  white noise  as a stationary stochastic process that has zero mean and constant
spectral density, is the generalized derivative of the Brownian motion (see, e.g., \cite[p.51, Theorem 3.14]{duan2015})
and often used to represent stochastic disturbances in many scenes.
However, the white noise could not invariably describe the  stochastic
disturbances emerging in practice due to the fact that its $\delta$-function correlation is
an idealization of the correlations of real processes which often have finite, or even long, correlation time \cite{colorednoise}.
A more realistic stochastic noise could be an exponentially correlated process,
which is known as aforementioned colored noise or Ornstein-Uhlenbeck process \cite{colorednoise,colourednoise2}.
\begin{remark}
This paper is the first effort on the event-triggered ESO for uncertain stochastic systems,
where both the designs and analyses are substantially different from the deterministic counterpart in \cite{eventESO}.
Therefore, similar to the plant considered in \cite{eventESO}, this paper only considers the most foundational
essential-integral-chain systems with matched uncertainties and disturbances which are widely considered
in the theory and application research of ADRC, to focus on the new designs and theoretical analyses of
linear and nonlinear event-triggered ESOs
when uncertainties and disturbances are stochastic. As its development to more general systems like
 lower triangle nonlinear ones \cite{zhaofrac} and MIMO ones
\cite{wuMIMOESO} which does not cause essential obstacle,  is not the main concern of this paper.
\end{remark}
\section{Linear and nonlinear event-triggered ESOs designs and main results}
Define the stochastic total disturbance (extended state) of system (\ref{system1.2}):
$x_{n+1}(t)=f(t,x(t),v_{1}(t),v_{2}(t))$,
which contains total nonlinear coupling effects of unmodeled dynamics, external deterministic disturbance, bounded noise and colored noise.
To estimate not only unmeasured state but also stochastic total disturbance of system (\ref{system1.2}), a
 linear event-triggered extended state observer (ESO)  is designed as follows:
 \begin{equation}\label{linearobserver}
\left\{\begin{array}{l}
\dot{\hat{x}}_{1}(t)=\hat{x}_{2}(t)+a_{1}r\left(y(t_{k})-\hat{x}_{1}(t)\right),
\cr
\dot{\hat{x}}_{2}(t)=\hat{x}_{3}(t)+a_{2}r^{2}\left(y(t_{k})-\hat{x}_{1}(t)\right),\cr
\hspace{1.2cm} \vdots \cr
 \dot{\hat{x}}_{n}(t)=\hat{x}_{n+1}(t)+a_{n}r^{n}\left(y(t_{k})-\hat{x}_{1}(t)\right)+u(t),  \cr
 \dot{\hat{x}}_{n+1}(t)=a_{n+1}r^{n+1}\left(y(t_{k})-\hat{x}_{1}(t)\right),
\end{array}\right.
\end{equation}
where $t\in [t_{k},t_{k+1}),\; k\in \mathbb{Z}^{+}$, $r$  is the tuning gain, and the parameters $a_{i}$'s are chosen such that the matrix $G$ defined in (\ref{matricf2}) is Hurwitz;
$\hat{x}(t):=(\hat{x}_{1}(t),\cdots,\hat{x}_{n+1}(t))$ contains the estimates of the state $x(t)$ and the stochastic
total disturbance $x_{n+1}(t)$;
$t_{k}$'s are stochastic execution times (stopping times) determined by the following
 ETM
\begin{equation}\label{lineartriggermechemi}
t_{k+1}=\inf\{t\geq t_{k}+\tau:\;|y(t)-y(t_{k})|\geq \theta r^{-(n+\frac{1}{2})} \},
\end{equation}
with $t_{1}=0$, $\tau$ being a positive constant specified in the following Theorem \ref{theorem3.1} and $\theta$ being any free positive tuning parameter.
The nonlinear event-triggered ESO via homogeneity is designed as
\begin{eqnarray}\label{nonlinearESO}
\left\{
\begin{array}{l}
 \dot{\hat{x}}_{1}(t)=\hat{x}_{2}(t)+\frac{a_{1}}{r^{n-1}}
\left \langle r^{n}(y(t_{k})-\hat{x}_{1}(t))\right\rangle^{\nu},
\cr
\dot{\hat{x}}_{2}(t)=\hat{x}_{3}(t)+\frac{a_{2}}{r^{n-2}}
\left\langle r^{n}(y(t_{k})-\hat{x}_{1}(t))\right\rangle^{2\nu-1},
\cr \hspace{1.2cm} \vdots \cr
\dot{\hat{x}}_{n}(t)=\hat{x}_{n+1}(t)+a_{n}
\left \langle r^{n}(y(t_{k})-\hat{x}_{1}(t))\right\rangle^{n\nu-(n-1)} \cr\hspace{1.3cm}+u(t), \cr
\dot{\hat{x}}_{n+1}(t)=a_{n+1}r
\left \langle r^{n}(y(t_{k})-\hat{x}_{1}(t))\right\rangle^{(n+1)\nu-n},
\end{array}\right.
\end{eqnarray}
where $t\in [t_{k},t_{k+1}),\; k\in \mathbb{Z}^{+}$, $r$  is the tuning gain,  the parameters $a_{i}$'s are chosen such that the matrix $G$ defined in (\ref{matricf2}) is Hurwitz,
and $\nu$ is a positive constant specified in the following Theorem \ref{theorem3.2};
$t_{k}$'s are stochastic execution times (stopping times) determined by the following ETM
\begin{equation}\label{nonlineartrigger}
t_{k+1}=\inf\{t\geq t_{k}+\tau^{*}:\;|y(t)-y(t_{k})|\geq \theta^{*} r^{-(n+\frac{1}{n\nu-(n-1)})} \},
\end{equation}
with $t_{1}=0$, $\tau^{*}$ being a positive constant specified in the following Theorem \ref{theorem3.2} and  $\theta^{*}$ being positive tuning parameter.

\begin{remark}
With regard to the triggering mechanism \dref{lineartriggermechemi} (or \dref{nonlineartrigger}),
 each inter-execution time $t_{k+1}-t_{k}$ is clearly not less than $\tau$ (or $\tau^{*}$),
so that the Zeno phenomenon can be naturally avoided. However, the convergence analysis of ESO
under these triggering mechanisms would be more complex.
  It should be also pointed out that in the designs of the triggering mechanisms  \dref{lineartriggermechemi}  and \dref{nonlineartrigger},
 only output  measurement are required to be monitored, not in real time, to evaluate the event-triggering condition.
 In addition, there exists a trade-off between the estimation accuracy and triggering frequency caused by
 the free tuning parameters $\theta$ and $\theta^{*}$ in aforementioned event-triggering mechanisms,
 that is, when $\theta$ or $\theta^{*}$ are tuned to be larger, the triggering frequency would be reduced yet the
 the estimation accuracy would also be reduced. This can be easily observed from the event-triggering mechanisms
 and the following main results.
\end{remark}
To guarantee the convergence of the linear and nonlinear event-triggered ESOs, the following assumptions are required.

{\bf  Assumption (A1).}
 The function $f$  has first-order continuous partial derivative and
second-order continuous partial derivative with respect to its arguments $(t,x)$ and
$(v_{1},v_{2})$, respectively, and there exist known constants  $\beta_{i}>0
\;(i=1,2,3,4)$ and non-negative functions $\gamma_{i}\in
C(\mathbb{R};\mathbb{R})\;(i=1,2)$ such that for all $t\ge 0$, $x \in \mathbb{R}^{n}$,
$v_{1}\in \mathbb{R}$, $v_{2}\in \mathbb{R}$, it holds that
\begin{eqnarray*}\label{2.40}
&&\hspace{-0.4cm}|f(t,x,v_{1},v_{2})|+|\frac{\partial f(t,x,v_{1},v_{2})}{\partial t}|
+|\frac{\partial f(t,x,v_{1},v_{2})}{\partial v_{1}}|
 \cr&&\hspace{-0.4cm}+ |\frac{\partial^{2} f(t,x,v_{1},v_{2})}{\partial v^{2}_{1}}|
\leq  \beta_{1}+\beta_{2}\|x\|+\beta_{3}\|v_{2}\|+\gamma_{1}(v_{1}) ; \cr&&\hspace{-0.4cm}
\sum^{n}_{i=1}|\frac{\partial f(t,x,v_{1},v_{2})}{\partial x_{i}}|+|\frac{\partial f(t,x,v_{1},v_{2})}{\partial v_{2}}|\cr&&\hspace{-0.4cm}+ |\frac{\partial^{2} f(t,x,v_{1},v_{2})}{\partial v^{2}_{2}}|
\leq \beta_{4}+\gamma_{2}(v_{1}).
\end{eqnarray*}

{\bf Assumption (A2).} The function $\sigma(t,\varsigma):[0,\infty)\times \mathbb{R} \rightarrow \mathbb{R}$
have first
order continuous partial derivative and second order continuous partial derivative with respect to
 $t$ and $\varsigma$, respectively, and there exists a known constant $\beta_{5}>0$, such that for all
  $t\geq 0$, $\varsigma\in \mathbb{R}$,
\begin{eqnarray*}\label{assumptiodn}
&&\hspace{-0.5cm}|\sigma(t,\varsigma)|+|\frac{\partial \sigma(t,\varsigma)}{\partial t}|
+|\frac{\partial \sigma(t,\varsigma)}{\partial \varsigma}|+
\frac{1}{2}|\frac{\partial^{2} \sigma(t,\varsigma)}{\partial \varsigma^{2}}|\leq \beta_{5}.
\end{eqnarray*}

{\bf  Assumption (A3).} The solution $x(t)$ and input $u(t)$ of system \dref{system1.2}
satisfies $\mathbb{E}\|x(t)\|^{2}+\mathbb{E}|u(t)|^{2}\leq M, \; \forall t\geq 0$,
for some known positive constant $M$.

\begin{remark}
Since the stochastic total disturbance $x_{n+1}(t)=f(t,x(t),v_{1}(t),v_{2}(t))$ is to be estimated in real time by ESO
and finally compensated in the feedback loop in the ADRC's framework, both the stochastic total disturbance
and its ``rate of change" should naturally be bounded guaranteed by Assumptions (A1)-(A3). In addition, the partial
derivatives in Assumption (A2) are assumed for the function $\sigma(\cdot,\cdot)$ defining the bounded noise, but not for the stochastic noise itself;
It can be seen that the conventional deterministic disturbance is just its special case by letting $v_{1}(t)=\sigma(t)$
that is the function with respect to the time variable $t$ only; For the stochastic counterpart, common bounded noises
like $\sin(t+B_{1}(t))$ and  $\cos(t+B_{1}(t))$ in practice (\cite{boundednoise1,boundednoise2}) are the concerning
ones satisfying Assumption (A2).
 \end{remark}

\begin{remark}
The rationality of Assumption (A3) can be further addressed as follows.
Firstly, it should be emphasized that this paper only investigates the
convergence of ESO for the open-loop system; And the boundedness of
state in Assumption (A3) is used for estimation of the state-dependent stochastic total disturbance,
which can be regarded as a ``slowly varying" condition of the open-loop system \dref{system1.2} besides the usual
structural one (i.e., exact observability). Secondly, if the stochastic total disturbance is state-independent or only the state
is estimated, it can be easily obtained from  the following proofs of main results that
we can get rid of this assumption, or refer to, e.g., \cite{zhaofrac}. Thirdly,
the state of many practical control systems is bounded like those in faults
diagnosis \cite{zhaofrac}.
Finally, ESO is the key component designed for the active anti-disturbance control objective,
so when the ESO-based feedback control is designed, i.e., estimation and control are performed simultaneously,
Assumption (A3) is not required because the closed-loop state is bounded in mean square sense, see for instance \cite{wu1}.
This topic will be further researched in the subsequent paper.
 \end{remark}

 Let  $Q\in \mathbb{R}^{(n+1)\times (n+1)}$ be the unique positive definite matrix solution of
the Lyapunov equation $QG+G^{\top}Q=-I_{n+1}$.
The mean square and almost sure convergence of the linear event-triggered ESO \dref{linearobserver}  under \dref{lineartriggermechemi}
 is summarized as the following Theorem \ref{theorem3.1}.
\begin{theorem}\label{theorem3.1}
Suppose that Assumptions (A1)-(A3) hold and let
$\tau=\epsilon r^{-(n+\frac{1}{2})}$ for any free tuning parameter $\epsilon>0$.
Then, for any initial values $x(0)\in \mathbb{R}^{n},\hat{x}(0)\in \mathbb{R}^{n+1}$, $v_{2}(0)\in \mathbb{R}$, $T>0$,
 $r\geq r^{*}:=\max\{1,\frac{\lambda^{2}_{\max}(Q)}{\zeta}\}$ with $\zeta$ being any constant satisfying $0<\zeta<1$ and all $i=1,\cdots,n+1$, the estimation errors of the linear event-triggered ESO \dref{linearobserver}  under \dref{lineartriggermechemi} satisfy
\begin{equation}\label{estimationerror}
\hspace{-1.9cm}\mbox{(i)}\;\;\;\;\;\mathbb{E}|x_{i}(t)-\hat{x}_{i}(t)|^2\leq \frac{\Theta}{r^{2n+3-2i}},
\end{equation}
uniformly in $t\in [T,\infty)$, where $\Theta>0$ is a constant independent of  $r$;
\begin{eqnarray}\label{df34e}
\hspace{-1.3cm}\mbox{(ii)}\;\;\;|x_{i}(t)-\hat{x}_{i}(t)|\leq \frac{\Theta_{\omega}}{r^{n+\frac{3}{2}-i}} \;\; \mbox{a.s.}
\end{eqnarray}
uniformly in $t\in [T,\infty)$, where $\Theta_{\omega}>0$ is a random variable independent of  $r$.
\end{theorem}

\begin{proof}
See  ``Proof of Theorem \ref{theorem3.1}" in Appendix A.
\end{proof}

In order to guarantee the convergence of the nonlinear event-triggered ESO \dref{nonlinearESO}  under \dref{nonlineartrigger},
the  boundedness of state is required to be slightly stronger than the one in Assumption (A3).

{\bf  Assumption (A4).} There exists some  known positive constant $N$ such that
the solution $x(t)$ and input $u(t)$ of system \dref{system1.2}
satisfies $\mathbb{E}\|x(t)\|^{p}+\mathbb{E}|u(t)|^{2}\leq N,\; \forall t\geq 0$,
where $p$ is any constant satisfying $p>2$.

 The almost sure convergence of the nonlinear event-triggered ESO \dref{nonlinearESO} under \dref{nonlineartrigger} is summarized as the following Theorem \ref{theorem3.2}.
\begin{theorem}\label{theorem3.2}
Suppose that Assumptions (A1), (A2) and (A4) hold, and let $\nu\in (\max\{1-\frac{p-2}{(p-2)n+p+1},
1-\frac{1}{2n-1}\},1)$ and
$\tau^{*}=\epsilon^{*}r^{-(n+\frac{1}{n\nu-(n-1)})}$ for any free tuning parameter $\epsilon^{*}>0$.
Then, there exists some $r^{*}\geq 1$ such that for any initial values $x(0)\in \mathbb{R}^{n},\hat{x}(0)\in \mathbb{R}^{n+1}$, $v_{2}(0)\in \mathbb{R}$,
 $r\geq r^{*}$ and all $i=1,\cdots,n+1$, the estimation errors of the nonlinear event-triggered ESO \dref{nonlinearESO} under \dref{nonlineartrigger} satisfy
\begin{eqnarray}\label{3fdfd}
|x_{i}(t)-\hat{x}_{i}(t)|\leq \frac{\Xi_{i,\omega}}{r^{n+1+\frac{(i-1)\nu-(i-2)}{\mu-1+\nu}-i}}\;\; \mbox{a.s.}
\end{eqnarray}
uniformly  in  $t\in [\varpi,\infty)$  , where $\varpi\geq 0$ is a random variable satisfying
$\mathbb{E}\varpi\leq T_{r}$ for some $r$-dependent positive constant $T_{r}$, $\Xi_{i,\omega}$'s are positive random variables independent of $r$, and
$\mu$ is any constant satisfying $\disp \max\{1,2[n\nu-(n-1)]\} <\mu <(2n-1)\nu-2n+3$.
\end{theorem}

\begin{proof}
See  ``Proof of Theorem \ref{theorem3.2}" in Appendix B.
\end{proof}

\begin{remark}\label{36fdf}
There are some conclusions that can be obtained directly from Theorem \ref{theorem3.1} and  Theorem \ref{theorem3.2}. Firstly, there is an anti-correlation
between the estimation errors and the tuning gain $r$. That is, the larger the tuning gain $r$ becomes,
the smaller the estimation errors are. Secondly, the estimation performance can be ensured in the transient sense, i.e., the estimation
errors become small after a time instant (may be stochastic). Finally, it can be obtained that $\frac{(i-1)\nu-(i-2)}{\mu-1+\nu}>\frac{1}{2}$ for
all $i=1,\cdots,n+1$,
thus it is seen from \dref{df34e} and \dref{3fdfd} that the estimation accuracy of the nonlinear event-triggered ESO  \dref{nonlinearESO} under \dref{nonlineartrigger}
is higher than the linear event-triggered ESO \dref{linearobserver}  under \dref{lineartriggermechemi}. However,
it is seen from
$\frac{1}{n\nu-(n-1)}>\frac{1}{2}$
that the triggering frequency required for
the  nonlinear event-triggered ESO \dref{nonlinearESO} under \dref{nonlineartrigger}  is higher than the one required for the linear
event-triggered ESO \dref{linearobserver}  under \dref{lineartriggermechemi}, which will also be illustrated by the following numerical simulations.
\end{remark}

\section{Numerical simulations}
In this section, to illustrate the effectiveness of the linear and nonlinear event-triggered ESOs, we take the following second-order uncertain stochastic systems as an numerical example:
\begin{equation}\label{systemsexam}
\left\{\begin{array}{l} \dot{x}_{1}(t)=x_{2}(t), \cr
\dot{x}_{2}(t)=-b_{1}x_{1}(t)-b_{2}x_{2}(t)+b_{3}\sin(b_{4}x_{1}(t)+b_{5}x_{2}(t))\cr
\hspace{1.2cm}+b_{6}\cos(b_{7}t+b_{8}B_{1}(t))+b_{9}v_{2}(t)+u(t),\cr
y(t)=x_{1}(t),
\end{array}\right.
\end{equation}
where $u(t)=\cos(b_{10}t)$, and $b_{i}\;(i=1,\cdots,10)$ are all unknown parameters with $b_{1},b_{2}>0$.
The stochastic total disturbance is $x_{3}(t):=-b_{1}x_{1}(t)-b_{2}x_{2}(t)+b_{3}\sin(b_{4}x_{1}(t)+b_{5}x_{2}(t))+b_{6}\cos(b_{7}t+b_{8}B_{1}(t))+b_{9}v_{2}(t)$.
It can be easily checked that Assumptions (A1-A4) are all satisfied.
For systems \dref{systemsexam}, the linear event-triggered ESO \dref{linearobserver}  under \dref{lineartriggermechemi}
is designed with $n=2$, $a_{1}=a_{2}=3$, $a_{3}=1$, $r=15$, $\epsilon=\theta=1$ and $\tau=15^{-2.5}$ as required;
The  nonlinear event-triggered ESO \dref{nonlinearESO} under \dref{nonlineartrigger} is designed with
$n=2$, $a_{1}=a_{2}=3$, $a_{3}=1$, $r=15$, $\epsilon^{*}=\theta^{*}=1$,  $p=3,\nu=\frac{6}{7}$ and
$\tau^{*}=15^{-\frac{17}{5}}$ as required. To observe and compare the estimation effects, initial values and unknown parameters are specified as
$x_{1}(0)=1,x_{2}(0)=-1,v_{2}(0)=0,\hat{x}_{1}(0)=\hat{x}_{2}(0)=\hat{x}_{3}(0)=0,$ $b_{1}=b_{2}=2$,
$b_{i}=1.5\;(i=3,\cdots,6)$, $b_{i}=2.5\;(i=7,\cdots,10)$, and $\alpha_{1}=\alpha_{2}=2$ in \dref{21equ}.
It is observed from Figure \ref{figure1linear} and Figure \ref{figure2nonlinear} that both linear and nonlinear
event-triggered ESOs have good estimation effects for $(x_{2}(t),x_{3}(t))$, while the estimation accuracy of the nonlinear event-triggered ESO is more satisfactory
than the linear counterpart. It is seen from Figure \ref{figure3tr} that the
 inter-execution times corresponding to ETM \dref{lineartriggermechemi} of the linear event-triggered ESO
 are overall larger than those corresponding to the nonlinear event-triggered ESO, where the number of execution times
 during $[0,20]$ is 121 and 1354, respectively. These are consistent with
 the main results as stated in Remark \ref{36fdf}.

\begin{figure}[ht]\centering
\subfigure[]
 {\includegraphics[width=4.3cm,height=3cm]{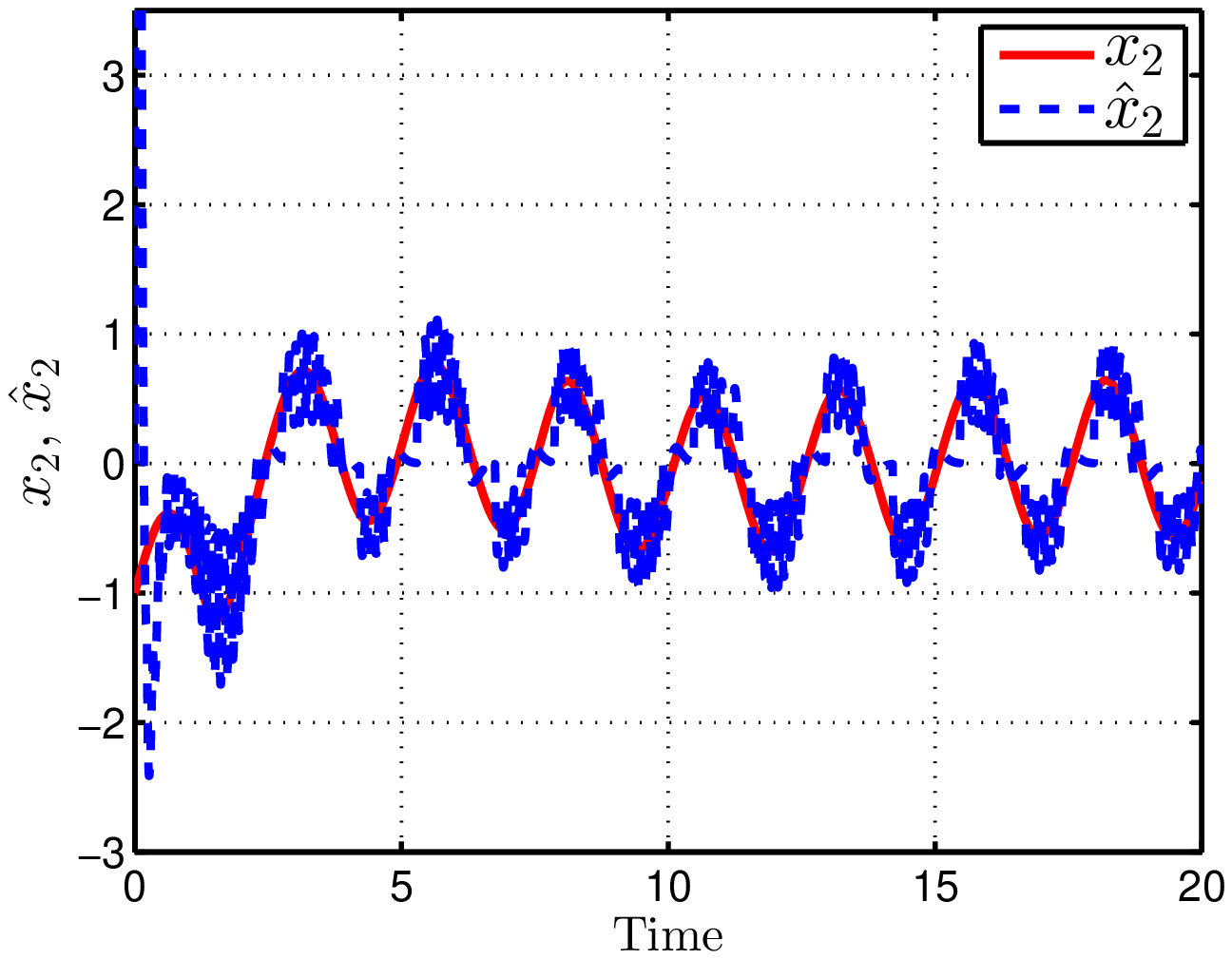}}
\subfigure[]
 {\includegraphics[width=4.3cm,height=3cm]{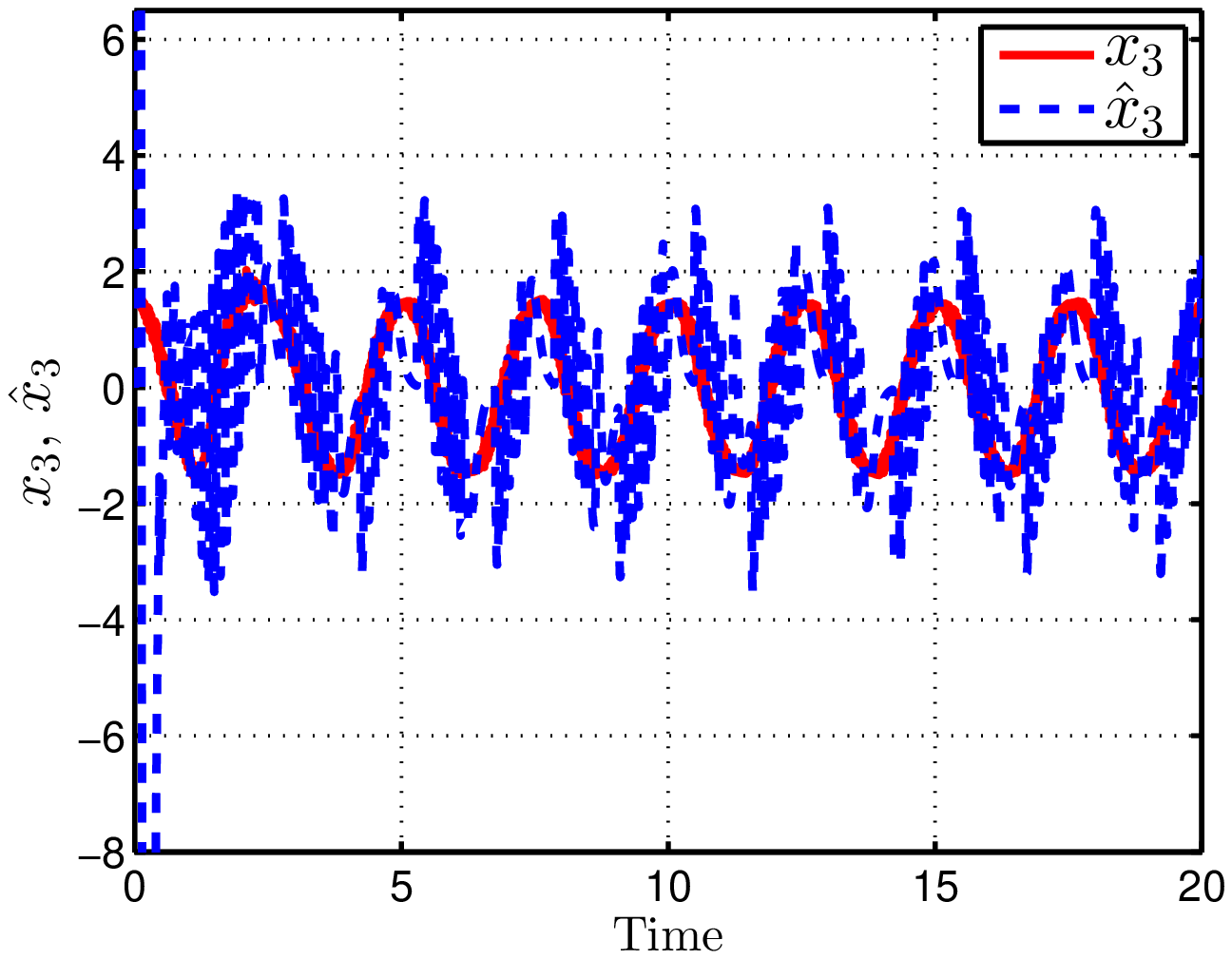}}
\caption{Estimation of $(x_{2}(t),x_{3}(t))$ by the linear event-triggered ESO.}\label{figure1linear}
\end{figure}

\begin{figure}[ht]\centering
\subfigure[]
 {\includegraphics[width=4.3cm,height=3cm]{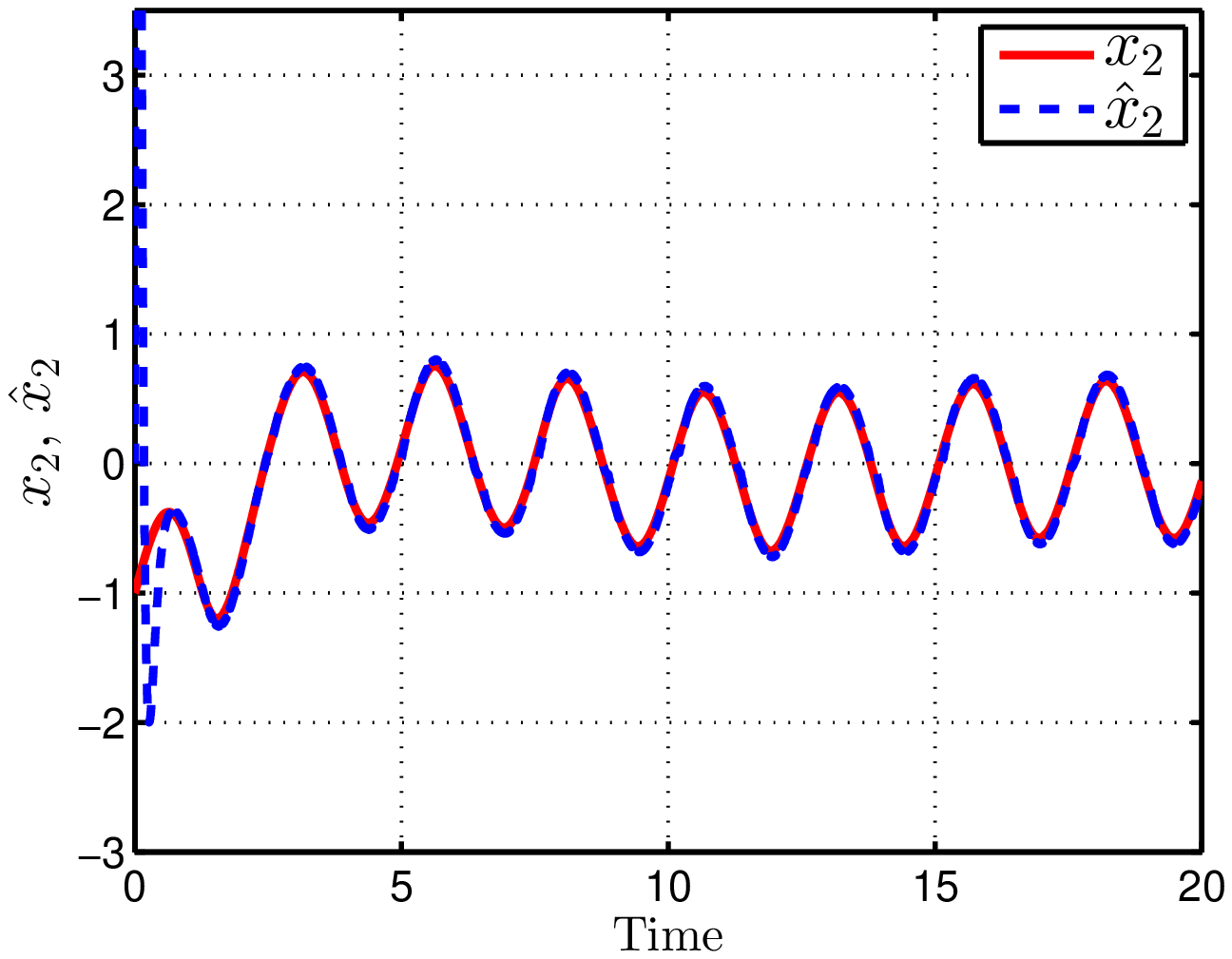}}
\subfigure[]
 {\includegraphics[width=4.3cm,height=3cm]{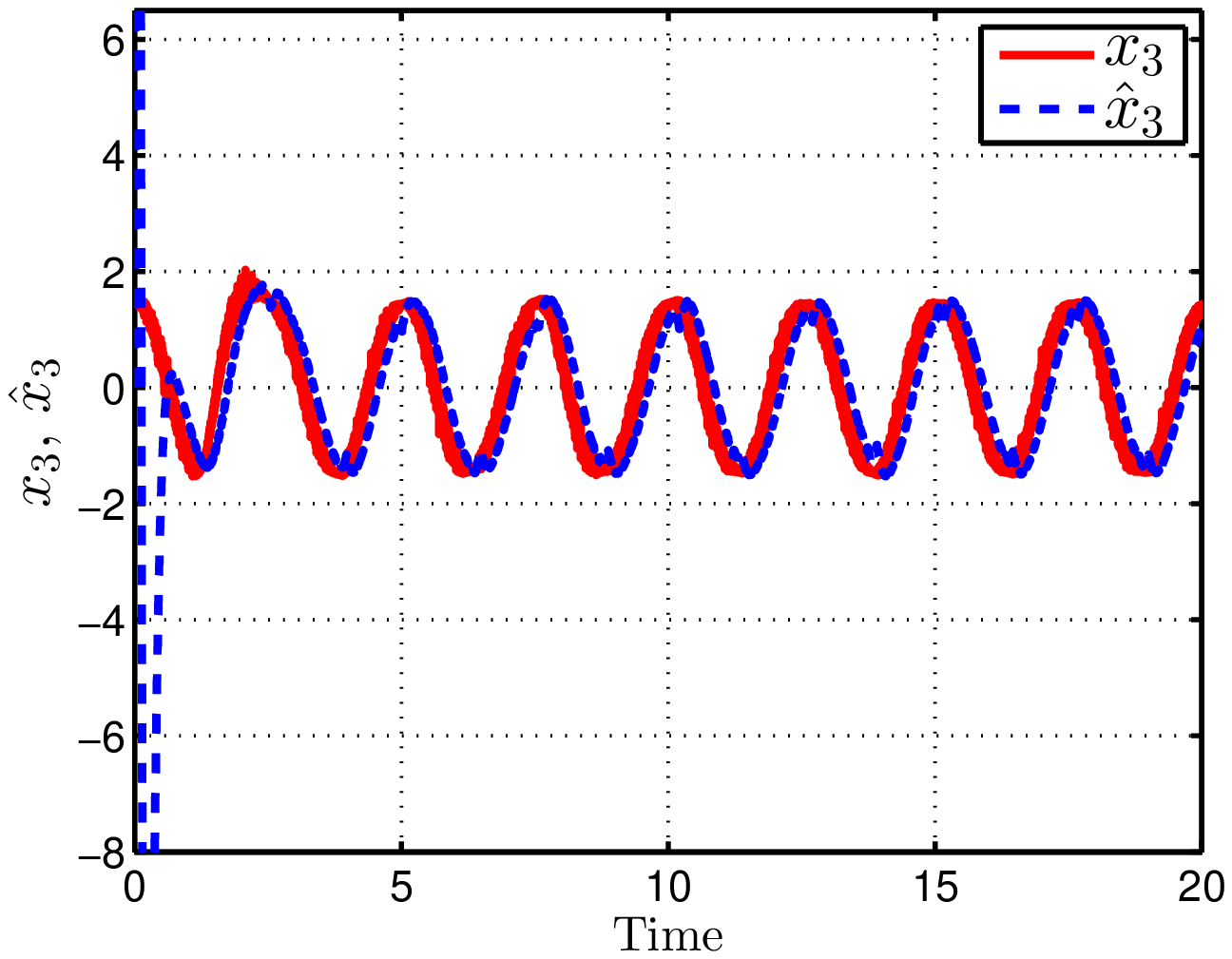}}
\caption{Estimation of $(x_{2}(t),x_{3}(t))$ by the nonlinear event-triggered ESO.}\label{figure2nonlinear}
\end{figure}

\begin{figure}[ht]\centering
\subfigure[Corresponding to ETM \dref{lineartriggermechemi}]
 {\includegraphics[width=4.3cm,height=3cm]{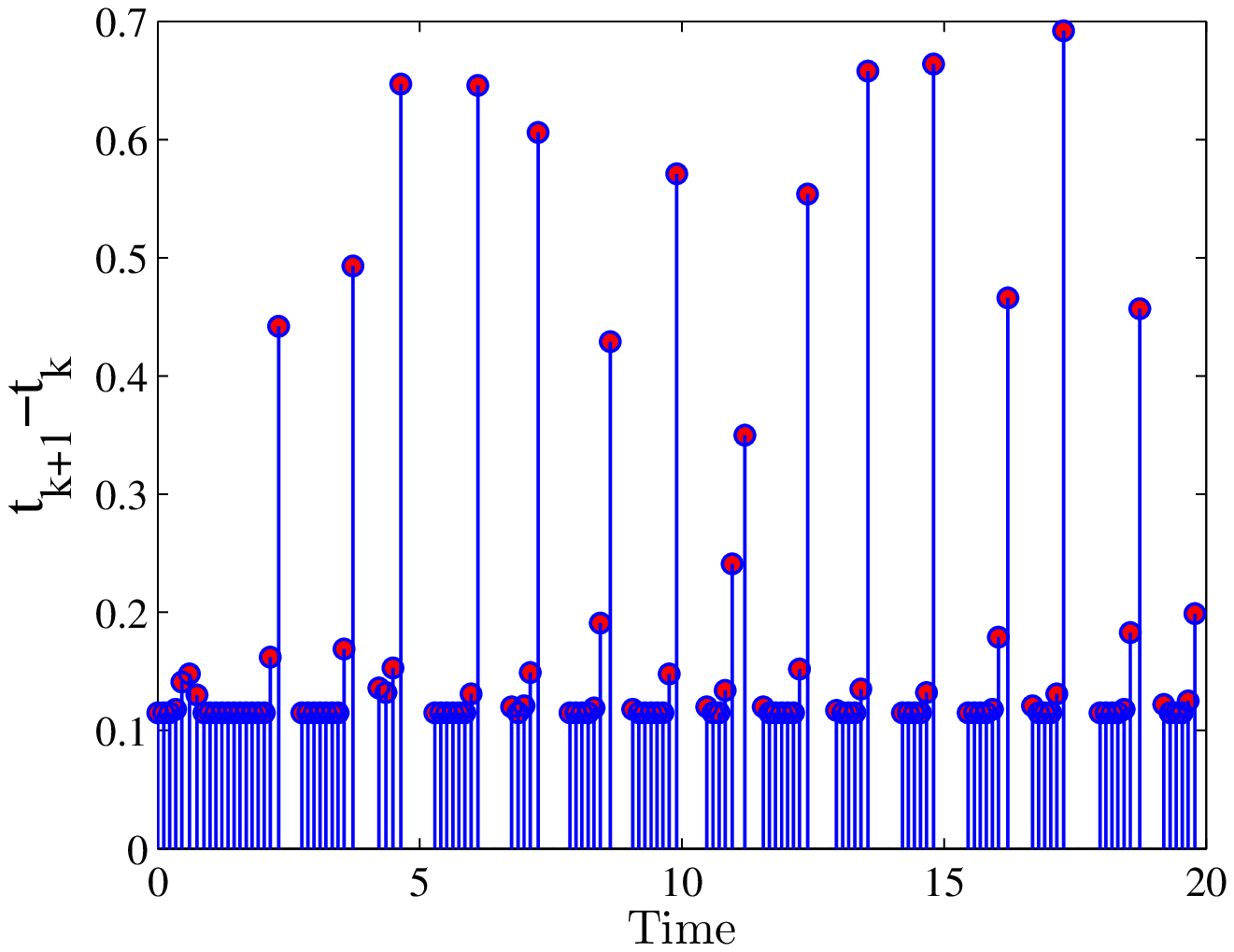}}
\subfigure[Corresponding to ETM \dref{nonlineartrigger}]
 {\includegraphics[width=4.3cm,height=3cm]{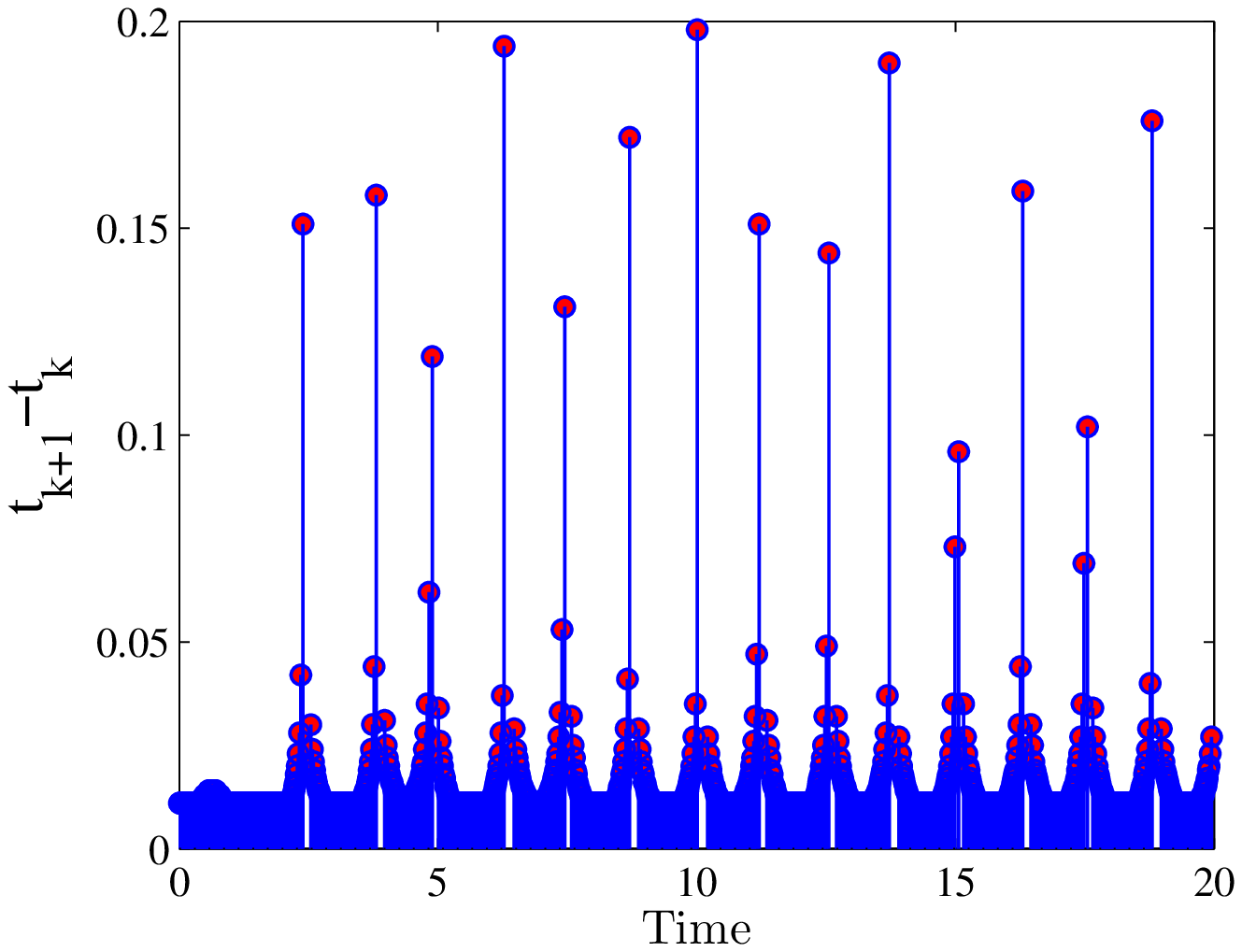}}
\caption{Inter-execution times corresponding to ETMs \dref{lineartriggermechemi} and \dref{nonlineartrigger}
for respective linear and nonlinear event-triggered ESOs.}\label{figure3tr}
\end{figure}

\section{Concluding remarks}\label{Se5}
This paper investigates both the linear and nonlinear event-triggered
extended state observers (ESOs) for a class of uncertain stochastic systems
driven by bounded and colored noises. Two event-triggering mechanisms with dwell time
are proposed for the designs of the event-triggered ESOs, which guarantee a positive minimum inter-event time for
every sample path solution of the stochastic systems to avoid directly the Zeno phenomenon.
Not only the mean square convergence but also the almost sure one of the estimation errors
of  unmeasured state and stochastic total disturbance
 are given with rigorous theoretical proofs.
The theoretical results also show that the nonlinear event-triggered ESO has higher estimation accuracy
but higher triggering frequency than the linear one. A successive interesting problem
to be further developed could be the event-triggered
active disturbance rejection control (ADRC) for the uncertain stochastic systems based on the  event-triggered ESOs.

\begin{center}
{\bf  APPENDIX A: Proof of Theorem \ref{theorem3.1}}
\end{center}


Define error variables as follows: $\eta_{i}(t)=r^{n+1-i}(x_{i}(t)-\hat{x}_{i}(t))\;(i=1,\cdots,n+1)$,
$\eta(t)=(\eta_{1}(t),\cdots,\eta_{n+1}(t))$.
For any fixed $t\in[0,\infty)$, the last execution time before $t$ can be expressed as
$\varrho_{t}=\max\{t_{k}:t_{k}\leq t, k\in \mathbb{Z}^{+}\}$.
Therefore, $y(t_{k})-y(t),\;t\in [t_{k},t_{k+1})$ can be expressed as $y(\varrho_{t})-y(t)$.
Set $\kappa(t)=r^{n}(y(\varrho_{t})-y(t))$.
Since $t_{k+1}-t_{k}\geq \tau$, for almost every sample path, there
are at most $[\frac{t}{\tau}]+1$ execution times before $t$. Define
$\Omega_{k}=\{\varrho_{t}=t_{k}\}, \; \Omega_{k,\tau}=\{\varrho_{t}=t_{k}\;\mbox{and}\; t\leq t_{k}+\tau\}, k=1,\cdots,[\frac{t}{\tau}]+1$.
It can be obtained that $\Omega$ can be expressed as the union of a set of mutually disjoint subset as
$\Omega=\bigcup^{[\frac{t}{\tau}]+1}_{k=1} \Omega_{k}$ for each $t\in[0,\infty)$.
Applying It\^{o}'s formula to $f(t,x(t),v_{1}(t),v_{2}(t))$ with respect to
$t$ along system \dref{system1.2} and \dref{21equ}, it is obtained that
\begin{eqnarray}\label{53fd}
&&\hspace{-0.7cm} dx_{n+1}(t)=\frac{\partial f(t,x(t),v_{1}(t),v_{2}(t))}{\partial
t}dt \cr&&\hspace{-0.7cm}+ \sum\limits^{n-1}_{i=1}\frac{\partial
f(t,x(t),v_{1}(t),v_{2}(t))}{\partial x_{i}}x_{i+1}(t)dt
  \cr  &&\hspace{-0.7cm}+\frac{\partial f(t,x(t),v_{1}(t),v_{2}(t))}{\partial
x_{n}}[f(t,x(t),v_{1}(t),v_{2}(t))+u(t)]dt \cr &&\hspace{-0.7cm}+\frac{\partial
f(t,x(t),v_{1}(t),v_{2}(t))}{\partial
v_{1}}[\frac{\partial \sigma(t,B_{1}(t))}{\partial t}+\frac{1}{2}\frac{\partial^{2} \sigma(t,B_{1}(t))}{\partial \varsigma^{2}}]dt \cr&&\hspace{-0.7cm}+
\frac{1}{2}\frac{\partial^{2}
f(t,x(t),v_{1}(t),v_{2}(t))}{\partial v^{2}_{1}}(\frac{\partial \sigma(t,B_{1}(t))}{\partial \varsigma})^{2}dt\cr&&\hspace{-0.7cm}-
 \frac{\partial f(t,x(t),v_{1}(t),v_{2}(t))}{\partial v_{2}}\alpha_{1}v_{2}(t)dt\cr&&\hspace{-0.7cm}+\frac{\partial^{2} f(t,x(t),v_{1}(t),v_{2}(t))}{\partial v^{2}_{2}}\alpha^{2}_{1}\alpha_{2}dt\cr&&\hspace{-0.7cm}
 +\frac{\partial f(t,x(t),v_{1}(t),v_{2}(t))}{\partial v_{1}}\frac{\partial \sigma(t,B_{1}(t))}{\partial \varsigma}dB_{1}(t) \cr&&\hspace{-0.7cm}+
 \frac{\partial f(t,x(t),v_{1}(t),v_{2}(t))}{\partial v_{2}}\alpha_{1}\sqrt{2\alpha_{2}}dB_{2}(t)\cr&&\hspace{-0.7cm}=:\Delta_{1}(t)dt+\Delta_{2}(t)dB_{1}(t)+\Delta_{3}(t)dB_{2}(t),
\end{eqnarray}
where $\varsigma$ represents the second argument of the function $\sigma(\cdot,\cdot)$. By Assumptions (A1)-(A3) and
the mean square boundedness of the colored noise $v_{2}(t)$, there exist known positive constants $C_{1}$ and $C_{2}$ such that
\begin{eqnarray}\label{55b}
 \sup_{t\geq 0}\mathbb{E}[|\Delta_{1}(t)|^{2}+|\Delta_{2}(t)|^{2}]\leq C_{1}, \;
\sup_{t\geq 0}|\Delta_{3}(t)|^{2} \leq C_{2}.
\end{eqnarray}
A direct computation shows that the error variable $\eta(t)$ satisfies the following It\^{o}-type stochastic differential equation
\begin{eqnarray} \label{system2.13}
\left\{\begin{array}{l}
d\eta_{1}(t)=r[\eta_{2}(t)-a_{1}\eta_{1}(t)]dt-a_{1}r\kappa(t)dt,\cr
d\eta_{2}(t)=r[\eta_{3}(t)-a_{2}\eta_{1}(t)]dt-a_{2}r\kappa(t)dt, \cr\hspace{1.2cm}\vdots\cr
d\eta_{n}(t)=r[\eta_{n+1}(t)-a_{n}\eta_{1}(t)]dt-a_{n}r\kappa(t)dt,
 \cr
d\eta_{n+1}(t)=-ra_{n+1}\eta_{1}(t)dt-a_{n+1}r\kappa(t)dt+\Delta_{1}(t)dt\cr \hspace{1.8cm}+\Delta_{2}(t)dB_{1}(t)+\Delta_{3}(t)dB_{2}(t).
\end{array}\right.
\end{eqnarray}
Define the Lyapunov function  $V:\mathbb{R}^{n+1}\rightarrow
\mathbb{R}$ by $V(\eta)=\eta Q \eta^{\top}$ for $\eta \in \mathbb{R}^{n+1}$.

Apply It\^{o}'s formula to  $V(\eta(t))$ with respect to
$t$ along system (\ref{system2.13}),
and by
\dref{55b}  and Young's inequality, for any $r\geq r^{*}$, it can be obtained that
\begin{eqnarray}\label{34fd2}
&&\hspace{-0.8cm} \frac{d\mathbb{E}V(\eta(t))}{dt} \cr&&\hspace{-0.8cm} \leq
-r\mathbb{E}\|\eta(t)\|^{2}+2r\lambda_{\max}(Q)\sum^{n+1}_{i=1}|a_{i}|\sum^{[\frac{t}{\tau}]+1}_{k=1}\mathbb{E}(\|\eta(t)\||\kappa(t)|\mathbb{I}_{\Omega_{k}})
\cr&&\hspace{-0.5cm}+\lambda^{2}_{\max}(Q)\mathbb{E}\|\eta(t)\|^{2}+C_{1}+\lambda_{\max}(Q)(C_{1}+C_{2})\cr&&\hspace{-0.8cm}
\leq -\xi_{0}r\mathbb{E}V(\eta(t))+\mu_{0} r\sum^{[\frac{t}{\tau}]+1}_{k=1}\mathbb{E}[\kappa^{2}(t)\mathbb{I}_{\Omega_{k}}]
+C_{1}\cr&&\hspace{-0.5cm}+\lambda_{\max}(Q)(C_{1}+C_{2}),
\end{eqnarray}
where $\mu_{0}>0$ is chosen such that $\xi_{0}:=\frac{1-\zeta}{\lambda_{\max}(Q)}-\frac{\lambda_{\max}(Q)(\sum^{n+1}_{i=1}|a_{i}|)^{2}}{\mu_{0}}>0.$
By the triggering mechanism \dref{lineartriggermechemi}, we have
$\sum^{[\frac{t}{\tau}]+1}_{k=1}\mathbb{E}[\kappa^{2}(t)\mathbb{I}_{\Omega_{k}\setminus\Omega_{k,\tau}}]\leq \frac{\theta^{2}}{r},\; \forall t\geq 0.$
It follows from Assumption (A3) and $\tau=\epsilon r^{-(n+\frac{1}{2})}$  that
\begin{eqnarray}\label{56equationsdd}
&&\hspace{-0.5cm} \sum^{[\frac{t}{\tau}]+1}_{k=1}\mathbb{E}[\kappa^{2}(t)\mathbb{I}_{\Omega_{k,\tau}}]
=r^{2n}\sum^{[\frac{t}{\tau}]+1}_{k=1}\mathbb{E}\left[\left(\int^{t}_{t_{k}}x_{2}(s)ds\right)^{2}\mathbb{I}_{\Omega_{k,\tau}}\right] \cr&& \hspace{-0.5cm}
\leq r^{2n}\tau \int^{t}_{t-\tau}\mathbb{E}x^{2}_{2}(s)ds \leq \frac{M\epsilon^{2}}{r},\; \forall t\geq \tau,
\end{eqnarray}
and it can be easily obtained that $\mathbb{E}[\kappa^{2}(t)\mathbb{I}_{\Omega_{1,\tau}}]\leq \frac{M\epsilon^{2}}{r}$ also holds for $t\in [0,\tau)$ where $\varrho_{t}=0$ a.s. For any $t\in [T,\infty)$, it holds that
$e^{-\xi_{0}rt}\mathbb{E}V(\eta(0))\leq \lambda_{\max}(Q) e^{-\xi_{0}rT}\sum^{n+1}_{i=1}r^{2(n+1-i)}\mathbb{E}|x_{i}(0)-\hat{x}_{i}(0)|^{2} \leq \frac{\Theta_{1}}{r}$,
for some $r$-independent constant $\Theta_{1}:=\sup_{r\geq r^{*}}\lambda_{\max}(Q) e^{-\xi_{0}rT}\sum^{n+1}_{i=1}r^{2n+3-2i}\mathbb{E}|x_{i}(0)-\hat{x}_{i}(0)|^{2}$
 which is finite since $\lim_{r\rightarrow \infty}e^{-\xi_{0}rT}\sum^{n+1}_{i=1}r^{2n+3-2i}=0$. Thus, for any $r\geq r^{*}$ and $t\in [T,\infty)$, we have
\begin{eqnarray*}
&& \hspace{-0.7cm}\mathbb{E}V(\eta(t))\cr&& \hspace{-0.7cm}\leq e^{-\xi_{0}rt}\mathbb{E}V(\eta(0))
+\mu_{0} r\int^{t}_{0}e^{-\xi_{0}r(t-s)}\sum^{[\frac{s}{\tau}]+1}_{k=1}\bigg\{\mathbb{E}[\kappa^{2}(s)\mathbb{I}_{\Omega_{k,\tau}}]\cr&&\hspace{-0.7cm}+\mathbb{E}[\kappa^{2}(s)\mathbb{I}_{\Omega_{k}\setminus\Omega_{k,\tau}}]\bigg\}ds
+\frac{C_{1}+\lambda_{\max}(Q)(C_{1}+C_{2})}{\xi_{0}r}
\leq \frac{\Theta_{2}}{r},
\end{eqnarray*}
where $\Theta_{2}:=\Theta_{1}+\frac{\mu_{0}(\theta^{2}+M\epsilon^{2})}{\xi_{0}}+\frac{C_{1}+\lambda_{\max}(Q)(C_{1}+C_{2})}{\xi_{0}}$.
This further yields that, for all $i=1,\cdots,n+1$,
\begin{eqnarray}\label{514fdf}
&&\hspace{-0.5cm}\mathbb{E}|x_{i}(t)-\hat{x}_{i}(t)|^{2}\leq \frac{1}{r^{2n+2-2i}}\mathbb{E}|\eta_{i}(t)|^{2}
\cr&&\hspace{-0.5cm}\leq \frac{1}{\lambda_{\min}(Q)r^{2n+2-2i}}\mathbb{E}V(\eta(t))\leq \frac{\Theta}{r^{2n+3-2i}},
\end{eqnarray}
uniformly in $t\in [T,\infty)$, where $\Theta:=\frac{\Theta_{2}}{\lambda_{\min}(Q)}$.
In addition, by \dref{514fdf} and Chebyshev's inequality (\cite[p.5]{mao}), it holds that
$P\{|x_{i}(t)-\hat{x}_{i}(t)|\geq m\sqrt{\Theta}(\frac{1}{r})^{\frac{2n+3-2i}{2}}\} \leq \frac{1}{m^{2}}$
uniformly in $t\in [T,\infty)$ and for $m\in \mathbb{Z}^{+}$. From the Borel-Cantelli's lemma (\cite[p.7]{mao}), for almost all $\omega\in \Omega$, there
exists a random variable $m^{*}(\omega)$ such that whenever $m\geq m^{*}(\omega)$, we have
\begin{eqnarray}
|x_{i}(t)-\hat{x}_{i}(t)|\leq \frac{ m\sqrt{\Theta}}{r^{n+\frac{3}{2}-i}}, \; t\in [T,\infty).
\end{eqnarray}
Set $\Theta_{\omega}=m^{*}(\omega)\sqrt{\Theta}> 0$ which is an $r$-independent random variable.
This completes the proof.

\begin{center}
{\bf  APPENDIX B: Proof of Theorem \ref{theorem3.2}}
\end{center}
The following $\eta(t)$, $\varrho_{t}$ and $\kappa(t)$ are defined as those in
 Appendix A.
Similarly, for almost every sample path, there
are at most $[\frac{t}{\tau^{*}}]+1$ execution times before $t$. Define
 $\Omega^{*}_{k}=\{\varrho_{t}=t_{k}\}, \; \Omega^{*}_{k,\tau^{*}}=\{\varrho_{t}=t_{k}\;\mbox{and}\; t\leq t_{k}+\tau^{*}\},k=1,\cdots,[\frac{t}{\tau^{*}}]+1$,
and then $\Omega=\bigcup^{[\frac{t}{\tau^{*}}]+1}_{k=1} \Omega^{*}_{k}$
for each $t\in[0,\infty)$. By simple analysis, it can be easily obtained that
$|\langle \theta_{1}\rangle^{\vartheta}-\langle \theta_{2}\rangle^{\vartheta}|\leq 2^{1-\vartheta}|\theta_{1}-\theta_{2}|^{\vartheta}$
for all $\theta_{1},\theta_{2}\in \mathbb{R}$ and $\vartheta\in (0,1)$.
For $i=1,\cdots,n+1$, we set $\delta_{i}(t)=\langle \eta_{1}(t)\rangle^{i\nu-(i-1)}-\langle\eta_{1}(t)+\kappa(t)\rangle^{i\nu-(i-1)}$.
It then follows that $|\delta_{i}(t)|\leq 2^{i(1-\nu)}|\kappa(t)|^{i\nu-(i-1)},\;\forall t\geq 0$.

A direct computation shows that the error variable $\eta(t)$ satisfies the following It\^{o}-type stochastic differential equation
\begin{equation}\label{errsystes}
\left\{\begin{array}{l}
d\eta_{1}(t)=[r(\eta_{2}(t)-a_{1}\langle\eta_{1}(t)\rangle^{\nu})+a_{1}r\delta_{1}(t)]dt,\cr
d\eta_{2}(t)=[r(\eta_{3}(t)-a_{2}\langle\eta_{1}(t)\rangle^{2\nu-1})+a_{2}r\delta_{2}(t)]dt, \cr \hspace{1.2cm}\vdots\cr
d\eta_{n}(t)=[r(\eta_{n+1}(t)-a_{n}\langle\eta_{1}(t)\rangle^{n\nu-(n-1)})+a_{n}r\delta_{n}(t)]dt,
 \cr
d\eta_{n+1}(t)=-ra_{n+1}\langle\eta_{1}(t)\rangle^{(n+1)\nu-n}dt+a_{n+1}r\delta_{n+1}(t)dt\cr\hspace{1.8cm}+\Delta_{1}(t)dt+\Delta_{2}(t)dB_{1}(t)+\Delta_{3}(t)dB_{2}(t),
\end{array}\right.
\end{equation}
where $\Delta_{i}(t)\;(i=1,2,3)$ are defined as those in \dref{53fd} satisfying (\ref{55b}).
As a consequence of Lemmas \ref{lemma2.2}-\ref{lemma2.3}, system $\dot{\eta}(t)=\Phi(\eta(t))$
is globally finite-time stable, where $\Phi(\cdot)$ is defined in \dref{functionchi}.
Using Lemma \ref{fdf}, for any $\mu$ satisfying $\disp \max\{1,2w_{n+1}\} <\mu < d_{0}+2w_{n+1}$ with $d_{0}:=1-\nu>0$,
 we can conclude that there exists
a positive definite, radially unbounded function $V\in C^{1}(\mathbb{R}^{n+1};\mathbb{R})\cap C^{\infty}(\mathbb{R}^{n+1}\setminus\{0_{n+1}\};\mathbb{R})\ $
 such that $V$ is homogeneous of degree
$\mu$ with respect  to weights $\{w_{l}=(l-1)\nu-(l-2)\}_{l=1}^{n+1}$,  and
the Lie derivative of $V(\eta)$ along the vector field $\Phi$:
\begin{eqnarray}\label{424s}
&&\hspace{-0.8cm}L_{\Phi}V(\eta):=\sum^{n+1}_{i=1}\frac{\partial V(\eta)}{\partial \eta_{i}}\Phi_{i}(\eta)=\sum_{i=1}^{n}\dfrac{\partial V(\eta)}{\partial
\eta_{i}}
\left(\eta_{i+1}\right.\cr&&\hspace{-0.8cm}\left.-a_{i}\langle\eta_{1}\rangle^{i\nu-(i-1)}\right)
-\dfrac{\partial V(\eta)}{\partial \eta_{n+1}}a_{n+1}
\langle\eta_{1}\rangle^{(n+1)\nu-n}
\end{eqnarray}
is negative definite. By the definition of  homogeneity in Definition 2 and the fact that
$V(\eta)$ is homogeneous of degree $\mu$ with respect  to weights $\{w_{l}\}_{l=1}^{n+1}$,
it can be easily obtained that $L_{\Phi}V(\eta)$, $\dfrac{\partial V(\eta_{1},\cdots,\eta_{n+1})}{\partial
\eta_{i}}$,  $\dfrac{\partial^{2} V(\eta_{1},\cdots,\eta_{n+1})}{\partial
\eta^{2}_{n+1}}$ and $|\eta_{i}|$ are homogeneous
of  degree $\mu-d_{0}$,  $\mu-w_{i}$, $\mu-2w_{n+1}$ and $w_{i}$ with respect  to weights
$\{w_{l}\}_{l=1}^{n+1}$, respectively. The above analysis together with Lemma \ref{lemma2.1} yields the following inequalities:
\begin{eqnarray}\label{LfV}
&&\hspace{-0.9cm} L_{\Phi}V(\eta)\le -c_{1}(V(\eta))^{\frac{\mu-d_{0}}{\mu}},
|\dfrac{\partial V(\eta)}{\partial \eta_{i}}| \le c_{2i}
(V(\eta))^{\frac{\mu-w_{i}}{\mu}}, \cr&& \hspace{-0.9cm}
|\eta_{i}| \le c_{3i}(V(\eta))^{\frac{w_{i}}{\mu}},\; \forall \eta\in \mathbb{R}^{n+1},
 \cr&& \hspace{-0.9cm}|\dfrac{\partial^{2} V(\eta)}{\partial \eta^{2}_{n+1}}| \le c_{4}
(V(\eta))^{\frac{\mu-2w_{n+1}}{\mu}},  \forall \eta\in \mathbb{R}^{n+1}\setminus\{0_{n+1}\},
\end{eqnarray}
for some positive constants $c_{1},c_{2i},c_{3i},c_{4}$. For any fixed $t\in[0,\infty)$, the transient convergence
result (\ref{3fdfd})
holds directly for $\omega \in \Omega_{t}:=\{\|\eta(t)\|=0\}$, so we only need to analyze $\mathcal{L}V(\eta(t))$ for $\omega \in \Omega\setminus\Omega_{t}$.
It follows from (\ref{LfV}) that
$|\dfrac{\partial^{2} V(\eta(t))}{\partial \eta^{2}_{n+1}}| \le c_{4}
(V(\eta(t)))^{\frac{\mu-2w_{n+1}}{\mu}}$,
for  $\omega \in \Omega\setminus\Omega_{t}$. For convenience of the symbol, we proceed the proof by assuming that this inequality holds for almost all $\omega \in \Omega$ without loss of generality.
Since $\nu>1-\frac{1}{n+1}$,
$\frac{\mu-w_{n+1}}{\mu}<\frac{\mu-d_{0}}{\mu}$.
These together with the inequalities used repeatedly hereinbelow: $ab\leq \frac{1}{p_{0}}a^{p_{0}}+\frac{1}{q_{0}}b^{q_{0}}$ and $(\sum^{m}_{i=1}a_{i})^{\vartheta}\leq m^{\vartheta-1}\sum^{m}_{i=1}a^{\vartheta}_{i}$ with $\vartheta,p_{0},q_{0}>1$, $\frac{1}{p_{0}}+\frac{1}{q_{0}}=1$, $m\in \mathbb{Z}^{+}$, $a,b,a_{i}\geq 0$, yield that
\begin{eqnarray*}
&&\hspace{-0.6cm}  \mathcal{L}V(\eta(t))\cr&&\hspace{-0.6cm}   =
rL_{\Phi}V(\eta(t))+r\sum^{n+1}_{i=1}\dfrac{\partial V(\eta(t))}{\partial \eta_{i}}a_{i}\delta_{i}(t)
+\dfrac{\partial V(\eta(t))}{\partial \eta_{n+1}}\Delta_{1}(t)\cr&&\hspace{-0.6cm}+\frac{1}{2}\dfrac{\partial^{2} V(\eta(t))}{\partial \eta^{2}_{n+1}}(\Delta^{2}_{2}(t)+\Delta^{2}_{3}(t))
\cr&&\hspace{-0.6cm} \leq -c_{1}r(V(\eta(t)))^{\frac{\mu-d_{0}}{\mu}}+\sum^{n+1}_{i=1}(a_{i}c_{2i}2^{(1-\nu)i})^{\frac{\mu-d_{0}}{\mu-w_{i}}}\frac{\mu-w_{i}}{\mu-d_{0}}\cdot\cr&&\hspace{-0.6cm} (V(\eta(t)))^{\frac{\mu-d_{0}}{\mu}}
+\sum^{n+1}_{i=1}\sum^{[\frac{t}{\tau^{*}}]+1}_{k=1}r^{\frac{\mu-d_{0}}{w_{i}-d_{0}}}\frac{w_{i}-d_{0}}{\mu-d_{0}}\{|\kappa(t)|^{\frac{w_{i+1}(\mu-d_{0})}{w_{i}-d_{0}}}\cr&&\hspace{-0.6cm}\mathbb{I}_{\Omega^{*}_{k}}\}
+\frac{c_{2(n+1)}(\mu-w_{n+1})}{\mu-d_{0}}(V(\eta(t)))^{\frac{\mu-d_{0}}{\mu}}+c_{2(n+1)}\cdot\cr&&\hspace{-0.6cm}\frac{(w_{n+1}-d_{0})}{\mu-d_{0}}|\Delta_{1}(t)|^{\frac{\mu-d_{0}}{w_{n+1}-d_{0}}}
+\frac{c_{4}(\mu-2w_{n+1})}{2(\mu-d_{0})}(V(\eta(t)))^{\frac{\mu-d_{0}}{\mu}}\cr&&\hspace{-0.6cm}+\frac{c_{4}(2w_{n+1}-d_{0})}{2(\mu-d_{0})}
2^{\frac{\mu-2w_{n+1}}{2w_{n+1}-d_{0}}}[|\Delta_{2}(t)|^{\frac{2(\mu-d_{0})}{2w_{n+1}-d_{0}}}+C^{\frac{\mu-d_{0}}{2w_{n+1}-d_{0}}}_{2}],
\end{eqnarray*}
where $w_{n+2}:=(n+1)\nu-n$. Choose $r_{0}>0$ such that $\frac{c_{1}r_{0}}{2}-\sum^{n+1}_{i=1}(a_{i}c_{2i}2^{(1-\nu)i})^{\frac{\mu-d_{0}}{\mu-w_{i}}}\frac{\mu-w_{i}}{\mu-d_{0}}-\frac{c_{2(n+1)}(\mu-w_{n+1})}{\mu-d_{0}}-\frac{c_{4}(\mu-2w_{n+1})}{2(\mu-d_{0})}>0$.
Then, for any $r\geq r^{*}:=\max\{1, r_{0}\}$ and any $t\geq0$, it holds that
\begin{eqnarray}\label{524equations}
\mathcal{L}V(\eta(t))  \leq -\frac{c_{1}r}{2}(V(\eta(t)))^{\frac{\mu-d_{0}}{\mu}}+\Gamma_{1}+\sum^{4}_{i=2}\Gamma_{i}(t),
\end{eqnarray}
where we set
\begin{eqnarray}
&&\hspace{-0.6cm} \Gamma_{1}=\frac{c_{4}(2w_{n+1}-d_{0})}{2(\mu-d_{0})}
2^{\frac{\mu-2w_{n+1}}{2w_{n+1}-d_{0}}}C^{\frac{\mu-d_{0}}{2w_{n+1}-d_{0}}}_{2}, \cr&& \hspace{-0.6cm}
\Gamma_{2}(t)=\sum^{n+1}_{i=1}\sum^{[\frac{t}{\tau^{*}}]+1}_{k=1}r^{\frac{\mu-d_{0}}{w_{i}-d_{0}}}\frac{w_{i}-d_{0}}{\mu-d_{0}}\{|\kappa(t)|^{\frac{w_{i+1}(\mu-d_{0})}{w_{i}-d_{0}}}\mathbb{I}_{\Omega^{*}_{k}}\}, \cr&& \hspace{-0.6cm}
\Gamma_{3}(t)=\frac{c_{2(n+1)}(w_{n+1}-d_{0})}{\mu-d_{0}}|\Delta_{1}(t)|^{\frac{\mu-d_{0}}{w_{n+1}-d_{0}}}, \cr&& \hspace{-0.6cm}
\Gamma_{4}(t)=\frac{c_{4}(2w_{n+1}-d_{0})}{2(\mu-d_{0})}
2^{\frac{\mu-2w_{n+1}}{2w_{n+1}-d_{0}}}|\Delta_{2}(t)|^{\frac{2(\mu-d_{0})}{2w_{n+1}-d_{0}}}.
\end{eqnarray}
Define the stopping time as follows:
\begin{eqnarray*}
\hspace{-0.1cm}\varpi=\inf\{t\geq 0: V(\eta(t))\leq [\frac{4}{c_{1}r}(\Gamma_{1}+\sum^{4}_{i=2}\Gamma_{i}(t))]^{\frac{\mu}{\mu-d_{0}}}\}.
\end{eqnarray*}
If $V(\eta(t))\geq [\frac{4}{c_{1}r}(\Gamma_{1}+\sum^{4}_{i=2}\Gamma_{i}(t))]^{\frac{\mu}{\mu-d_{0}}}$ for some $t>\varpi$, then
$\frac{d\mathbb{E}V(\eta(t))}{dt}=\mathbb{E}\mathcal{L}V(\eta(t))\leq 0$. Thus, $\mathbb{E}V(\eta(t))\leq \mathbb{E}[\frac{4}{c_{1}r}(\Gamma_{1}+\sum^{4}_{i=2}\Gamma_{i}(t))]^{\frac{\mu}{\mu-d_{0}}}$
for all $t\in [\varpi,\infty)$.  Therefore, for all $i=1,\cdots,n+1$ and $t\in [\varpi,\infty)$, it follows from (\ref{LfV}) and H\"{o}lder's inequality that
\begin{eqnarray}\label{515equation}
&&\hspace{-0.7cm} \mathbb{E}|x_{i}(t)-\hat{x}_{i}(t)|=\frac{1}{r^{n+1-i}}\mathbb{E}|\eta_{i}(t)|
\leq \frac{c_{3i}}{r^{n+1-i}}\mathbb{E}(V(\eta(t)))^{\frac{w_{i}}{\mu}}
\cr&& \hspace{-0.7cm} \leq \frac{c_{3i}}{r^{n+1-i}}(\mathbb{E}(V(\eta(t)))^{\frac{w_{i}}{\mu}}
\leq (\frac{1}{r})^{n+1+\frac{w_{i}}{\mu-d_{0}}-i}c_{3i}(\frac{4}{c_{1}})^{\frac{w_{i}}{\mu-d_{0}}}\cr&&\hspace{-0.3cm} \times4^{\frac{d_{0}w_{i}}{(\mu-d_{0})\mu}}\left[\Gamma_{1}^{\frac{\mu}{\mu-d_{0}}}
+\sum^{4}_{i=2}\mathbb{E}[\Gamma_{i}(t)]^{\frac{\mu}{\mu-d_{0}}}\right]^{\frac{w_{i}}{\mu}}.
\end{eqnarray}
By $\tau^{*}=\epsilon^{*}r^{-(n+\frac{1}{w_{n+1}})}$ and Assumption (A4), we have
\begin{eqnarray}\label{528equation}
&&\hspace{-0.5cm} \mathbb{E}\{\sum^{n+1}_{i=1}\sum^{[\frac{t}{\tau^{*}}]+1}_{k=1}r^{\frac{\mu-d_{0}}{w_{i}-d_{0}}}\frac{w_{i}-d_{0}}{\mu-d_{0}}[|\kappa(t)|^{\frac{w_{i+1}(\mu-d_{0})}{w_{i}-d_{0}}}\mathbb{I}_{\Omega^{*}_{k,\tau}}]\}^{\frac{\mu}{\mu-d_{0}}}\cr &&\hspace{-0.5cm}
\leq (n+1)^{\frac{d_{0}}{\mu-d_{0}}}\sum^{n+1}_{i=1}\sum^{[\frac{t}{\tau^{*}}]+1}_{k=1}r^{(1+nw_{i+1})\frac{\mu}{w_{i}-d_{0}}}(\frac{w_{i}-d_{0}}{\mu-d_{0}})^{\frac{\mu}{\mu-d_{0}}}\cr &&\hspace{-0.1cm}\times\mathbb{E}((\int^{t}_{t_{k}}|x_{2}(s)|ds)^{\frac{w_{i+1}\mu}{w_{i}-d_{0}}}\mathbb{I}_{\Omega^{*}_{k,\tau}}) \cr&&\hspace{-0.5cm}
\leq (n+1)^{\frac{d_{0}}{\mu-d_{0}}}\sum^{n+1}_{i=1}r^{(1+nw_{i+1})\frac{\mu}{w_{i}-d_{0}}}(\frac{w_{i}-d_{0}}{\mu-d_{0}})^{\frac{\mu}{\mu-d_{0}}}
\cr &&\hspace{-0.1cm}\times(\int^{t}_{t-\tau^{*}}\mathbb{E}|x_{2}(s)|^{p}ds)^{\frac{w_{i+1}\mu}{p(w_{i}-d_{0})}}
 \tau^{*\frac{w_{i+1}\mu}{q(w_{i}-d_{0})}} \cr&&\hspace{-0.5cm}\leq (n+1)^{\frac{d_{0}}{\mu-d_{0}}}\sum^{n+1}_{i=1} r^{(1+nw_{i+1})\frac{\mu}{w_{i}-d_{0}}}(\frac{w_{i}-d_{0}}{\mu-d_{0}})^{\frac{\mu}{\mu-d_{0}}}\cr &&\hspace{-0.2cm}\times N^{\frac{w_{i+1}\mu}{p(w_{i}-d_{0})}}\tau^{*\frac{w_{i+1}\mu}{w_{i}-d_{0}}}
\leq M_{1}, \; \forall t\geq \tau^{*},
\end{eqnarray}
where $M_{1}=(n+1)^{\frac{d_{0}}{\mu-d_{0}}}\sum^{n+1}_{i=1}(\frac{w_{i}-d_{0}}{\mu-d_{0}})^{\frac{\mu}{\mu-d_{0}}}N^{\frac{w_{i+1}\mu}{p(w_{i}-d_{0})}}\epsilon^{*\frac{w_{i+1}\mu}{w_{i}-d_{0}}}$
and $q$ satisfies $\frac{1}{p}+\frac{1}{q}=1$, and it follows easily that the inequality \dref{528equation} also holds for $t\in [0,\tau^{*})$ where $\varrho_{t}=0$ a.s.
By the event-triggering mechanism (\ref{nonlineartrigger}), for all $t\geq 0$, it holds that
\begin{eqnarray}\label{529equations}
&&\hspace{-0.5cm}\mathbb{E}\{\sum^{n+1}_{i=1}\sum^{[\frac{t}{\tau^{*}}]+1}_{k=1}r^{\frac{\mu-d_{0}}{w_{i}-d_{0}}}\frac{w_{i}-d_{0}}{\mu-d_{0}}[|\kappa(t)|^{\frac{w_{i+1}(\mu-d_{0})}{w_{i}-d_{0}}}\mathbb{I}_{\Omega^{*}_{k}\setminus\Omega^{*}_{k,\tau}}]\}^{\frac{\mu}{\mu-d_{0}}}
\cr&&\hspace{-0.5cm} \disp \leq M_{2},
\end{eqnarray}
 where $M_{2}:= (n+1)^{\frac{d_{0}}{\mu-d_{0}}}\sum^{n+1}_{i=1}\theta^{*\frac{w_{i+1}\mu}{w_{i}-d_{0}}}(\frac{w_{i}-d_{0}}{\mu-d_{0}})^{\frac{\mu}{\mu-d_{0}}}.$
Thus, it follows from (\ref{528equation}) and (\ref{529equations}) that
\begin{eqnarray}
\mathbb{E}[\Gamma_{2}(t)]^{\frac{\mu}{\mu-d_{0}}}\leq 2^{\frac{d_{0}}{\mu-d_{0}}}(M_{1}+M_{2}),\; \forall t\geq 0.
\end{eqnarray}
 By $\nu>1-\frac{p-2}{(p-2)n+p+1}$ and $\mu<d_{0}+2w_{n+1}$, it follows that $\frac{\mu}{w_{n+1}-d_{0}}<p$.
Thus, it follows from Assumptions (A1), (A2) and (A4) that for all $t\geq 0$,
\begin{eqnarray*}
&&\hspace{-0.6cm}\mathbb{E}[\Gamma_{3}(t)]^{\frac{\mu}{\mu-d_{0}}}\leq [\frac{c_{2(n+1)}(w_{n+1}-d_{0})}{\mu-d_{0}}]^{\frac{\mu}{\mu-d_{0}}}\mathbb{E}|\Delta_{1}(t)|^{p}\leq M_{3}
\end{eqnarray*}
for some $M_{3}>0$. Similarly,
we also have $\frac{2\mu}{2w_{n+1}-d_{0}}<p$. Then, it follows from Assumptions (A1), (A2) and (A4) that for all $t\geq 0$,
\begin{eqnarray*}
&&\hspace{-0.6cm}\mathbb{E}[\Gamma_{4}(t)]^{\frac{\mu}{\mu-d_{0}}}\leq [\frac{c_{4}(2w_{n+1}-d_{0})}{2(\mu-d_{0})}
2^{\frac{\mu-2w_{n+1}}{2w_{n+1}-d_{0}}}]^{\frac{\mu}{\mu-d_{0}}}\mathbb{E}|\Delta_{2}(t)|^{p}\cr&&\hspace{1.4cm}\leq M_{4}
\end{eqnarray*}
for some $M_{4}>0$. These together with \dref{515equation}, yield that for all $i=1,\cdots,n+1$ and $t\in [\varpi,\infty)$,
\begin{eqnarray}
\mathbb{E}|x_{i}(t)-\hat{x}_{i}(t)|\leq \Xi_{i}\cdot(\frac{1}{r})^{n+1+\frac{w_{i}}{\mu-d_{0}}-i},
\end{eqnarray}
where we set $\Xi_{i}= c_{3i}(\frac{4}{c_{1}})^{\frac{w_{i}}{\mu-d_{0}}}4^{\frac{d_{0}w_{i}}{(\mu-d_{0})\mu}}
[\Gamma_{1}^{\frac{\mu}{\mu-d_{0}}}+2^{\frac{d_{0}}{\mu-d_{0}}}(M_{1}+M_{2})+M_{3}+M_{4}]^{\frac{w_{i}}{\mu}}$.
In addition, by Chebyshev's inequality (\cite[p.5]{mao}), it holds that
$P\{|x_{i}(t)-\hat{x}_{i}(t)|\geq m^{2}\Xi_{i}\cdot(\frac{1}{r})^{n+1+\frac{w_{i}}{\mu-d_{0}}-i}\} \leq \frac{1}{m^{2}}$
for all $m\in \mathbb{Z}^{+}$, $i=1,\cdots,n+1$ and $t\in [\varpi,\infty)$. From the Borel-Cantelli's lemma (\cite[p.7]{mao}), for almost all $\omega\in \Omega$
 and $t\in [\varpi,\infty)$, there
exists a random variable $m_{0}(\omega)$ such that whenever $m\geq m_{0}(\omega)$, we have
\begin{eqnarray}
|x_{i}(t)-\hat{x}_{i}(t)|\leq m^{2}\Xi_{i}\cdot(\frac{1}{r})^{n+1+\frac{w_{i}}{\mu-d_{0}}-i}.
\end{eqnarray}
Then \dref{3fdfd} is  obtained by letting $\Xi_{i,\omega}=m^{2}_{0}(\omega)\Xi_{i}>0$ which is a random variable independent of $r$.
Next we prove $\mathbb{E}\varpi\leq T_{r}$ for some $r$-dependent positive constant $T_{r}$. Define a series of stopping times as follows:
\begin{eqnarray*}
&&\hspace{-0.8cm}\varpi_{m}=\inf\{t\geq 0: V(\eta(t))\leq [\frac{4}{c_{1}r}(\Gamma_{1}+\sum^{4}_{i=2}\Gamma_{i}(t))]^{\frac{\mu}{\mu-d_{0}}}+\frac{1}{m}\cr&& \hspace{-0.8cm}  \mbox{or}\;  V(\eta(t))\geq [\frac{4}{c_{1}r}(\Gamma_{1}+\sum^{4}_{i=2}\Gamma_{i}(t))]^{\frac{\mu}{\mu-d_{0}}}+m\},\; m\in \mathbb{Z}^{+}.
\end{eqnarray*}
Set $\Upsilon(V)=\frac{4}{c_{1}r}\int^{V}_{0}\vartheta^{-\frac{\mu-d_{0}}{\mu}}d\vartheta, \; V\in [0,\infty)$.
Thus, for all $t\in [0,\varpi_{m})$, it holds that
 \begin{eqnarray}
&&\hspace{-0.6cm}\mathcal{L}\Upsilon(V(\eta(t)))\cr&&\hspace{-0.6cm}=\frac{4}{c_{1}r}(V(\eta(t)))^{-\frac{\mu-d_{0}}{\mu}}\mathcal{L}V(\eta(t))
-\frac{2(\mu-d_{0})}{c_{1}r\mu}(V(\eta(t)))^{-\frac{2\mu-d_{0}}{\mu}}\cr&&\hspace{-0.3cm}\times[(\frac{\partial V(\eta(t))}{\partial \eta_{n+1}}\Delta_{2}(t))^{2}+(\frac{\partial V(\eta(t))}{\partial \eta_{n+1}}\Delta_{3}(t))^{2}]\leq -1.
\end{eqnarray}
Set $t_{m}=m\wedge \varpi_{m}$ for $m\in \mathbb{Z}^{+}$. Then
 \begin{eqnarray}
&&\mathbb{E}[\Upsilon(V(\eta(t_{m})))]-\mathbb{E}[\Upsilon(V(\eta(0)))]
\cr&&=\mathbb{E}\int^{t_{m}}_{0}\mathcal{L}\Upsilon (V(\eta(s)))ds
\leq -\mathbb{E}[t_{m}]
\end{eqnarray}
which means that $\mathbb{E}[t_{m}]\leq \mathbb{E}[\Upsilon(V(\eta(0)))]$.
Passing to the limit as $m\rightarrow\infty$ and using the
Fatou's lemma, we obtain $\mathbb{E}\varpi\leq \frac{4\mu}{c_{1}d_{0}r}\mathbb{E}[V(\eta(0))]^{\frac{d_{0}}{\mu}}=:T_{r}<\infty$.
This completes the proof.
%
%

\ifCLASSOPTIONcaptionsoff
  \newpage
\fi



%

%






\end{document}